\documentclass[a4paper]{amsart}

\usepackage{amsmath}
\usepackage{amsfonts}
\usepackage{amssymb}
\newtheorem{theorem}{Theorem}[]
\newtheorem{proposition}{Proposition}[section]
\newtheorem{corollary}[proposition]{Corollary}
\newtheorem{lemma}[proposition]{Lemma}
\theoremstyle{definition}

\newtheorem*{acknowledgements}{Acknowledgements}

\newcommand{\R}{\mathbb{R}} % It denotes the set of real numbers
\newcommand{\C}{\mathbb{C}} % It denotes the set of complex numbers
\newcommand{\N}{\mathbb{N}} % It denotes the set of non-negative integers
\newcommand{\duality}[2]{\big\langle #1 | #2 \big\rangle} % It is used to denote the pairing of duality inside text
\DeclareMathOperator{\supp}{supp}

\author{Pedro Caro}
\author{Keith M. Rogers}
\title[Uniqueness for the Calder\'on problem]{Global uniqueness for the Calder\'on problem with Lipschitz conductivities}
\date{}
\thanks{The first author was partially supported by BERC 2014-2017 and  BCAM
Severo Ochoa SEV-2013-0323. The second author was partially supported by MTM2013-41780-P and ICMAT Severo Ochoa SEV-2011-0087 / SEV-2015-0554. Both authors were partially supported by the ERC starting grant 277778.}
\address{BCAM - Basque Center for Applied Mathematics, 48009 Bilbao, Spain and Ikerbasque, Basque Foundation for Science, 48011 Bilbao, Spain}
\email{pcaro@bcamath.org}
\address{Instituto de Ciencias Matem\'aticas CSIC-UAM-UC3M-UCM, 28049 Madrid, Spain} 
\email{keith.rogers@icmat.es}

\begin{document}

\begin{abstract} We prove uniqueness for Calder\'on's problem with Lipschitz conductivities in higher dimensions. Combined with the recent work of Haberman, who treated the three and four dimensional cases, this confirms a conjecture of Uhlmann. Our proof builds on the work of Sylvester and Uhlmann, Brown, and  Haberman and Tataru who proved uniqueness for $C^1$-conductivities and Lipschitz conductivities sufficiently close to the identity. 
\end{abstract}

\maketitle

%\tableofcontents
%\setcounter{tocdepth}{1}

\section{Introduction}

We consider the conductivity equation $\nabla \cdot (\gamma \nabla u) = 0$ on a bounded domain~$\Omega$ in $\R^n$ with $n\ge 3$. For each voltage potential placed on the boundary $\partial\Omega$, we are given the induced perpendicular current flux.
 In other words, we are given the Dirichlet-to-Neumann ({\small DN}) map $ \Lambda_\gamma$ 
formally defined by $$\Lambda_\gamma\,:\,u|_{\partial\Omega}\mapsto
\gamma\partial_\nu u |_{\partial\Omega}.$$
Here, $u$ solves the conductivity equation in $\Omega$ and $\partial_\nu$ is the outward normal derivative on the boundary $\partial\Omega$.
Then the  goal is
to recover  $\gamma$ from this information. 

Uhlmann conjectured  that if the conductivities $\gamma$ are bounded below by a positive constant 
 and belong to  $\mathrm{Lip}(\overline{\Omega})$, meaning that there is a constant $c > 0$ such that
\begin{align*}
|\gamma(x)| &\leq c,\qquad \forall\ x \in \overline{\Omega},\\
|\gamma(x) - \gamma(y)| & \leq c |x - y|,\qquad \forall\ x,y \in \overline{\Omega},
\end{align*}
then uniqueness should be guaranteed (see for example~\cite[Open Problem 1]{U0}).
That is to say there should be only one conductivity in the class for each {\small DN} map.
With this level of regularity the {\small DN}  map $ \Lambda_\gamma : H^{1/2} (\partial \Omega) \to H^{-1/2} (\partial \Omega) $ can be  defined via duality;
\[ \duality{\Lambda_\gamma f}{g} = \int_\Omega \gamma \nabla u \cdot \nabla v \,  \]
for any $ f, g \in H^{1/2} (\partial \Omega) $, where $ u,v \in H^1(\Omega) $ satisfy $ v|_{\partial \Omega} = g $ and  
\begin{equation*}
\left\{
\begin{aligned}
\nabla \cdot (\gamma \nabla u) &= 0\quad \mathrm{in}\ \Omega, \\
u|_{\partial \Omega} &= f.
\end{aligned}
\right.
\end{equation*}
For  sufficiently smooth conductivities, boundaries and solutions, this definition  corresponds with the previous one by integration by parts.

In the following theorem we improve on the result of Haberman and Tataru~\cite{HT} who proved uniqueness for $C^1$-conductivities and Lipschitz conductivities sufficiently close to the identity. 

\begin{theorem} \label{th:main_theorem} \sl Let $n\ge 3$ and consider $\Omega \subset \R^n$  a bounded domain with Lipschitz boundary. Let $\gamma_1,\gamma_2\in\mathrm{Lip}(\overline{\Omega})$ with $\gamma_1,\gamma_2 \ge c_0>0 $. Then $$\Lambda_{\gamma_1} = \Lambda_{\gamma_2}\quad \Rightarrow\quad \gamma_1 = \gamma_2.$$
\end{theorem}

Counterexamples for unique continuation problems (see for example~\cite{P63} or~\cite{W}) show a loss of rigidity of the solutions that  suggest that uniqueness might fail for conductivities in H\"older spaces. However 
 uniqueness was recently proved in~\cite{H} for conductivities in $W^{1,n}\cap L^\infty(\Omega)$, with $n=3$,~$4$, so the gradient of the conductivity need not be bounded. 
The two--dimensional problem seems to have different mathematical properties; see~\cite{AP} for the solution to  
the uniqueness problem and~\cite{ALP} for the limits of invisibility and visibility. Examples of invisibility can also be understood as counterexamples to uniqueness.

We now briefly describe the main steps in the proof of Theorem~\ref{th:main_theorem}, highlighting our own contribution at the end. Kohn and Vogelius~\cite{KV} proved that smooth conductivities and their derivatives can be recovered on the boundary from the {\small DN} map. This was extended by Alessandrini~\cite{Al90}, who uniquely determined Lipschitz conductivities on Lipschitz boundaries.  This allows us to extend the conductivities to the whole space, using the method of Whitney, in such a way that they are equal outside of $\Omega$. We then transform the conductivity equation to a  Schr\"odinger equation by writing $v = \gamma^{1/2} u$ and $q = \gamma^{-1/2} \Delta \gamma^{1/2}$, so that
\[ \nabla \cdot (\gamma \nabla u) = 0  \qquad \Leftrightarrow \qquad (-\Delta + q)v = 0.\quad \]
In fact, following Brown~\cite{B}, we interpret $q$ as a multiplication map defined via duality as
\begin{align*}
\big\langle q \phi, \psi \big\rangle &= - \int \nabla \gamma^{1/2} \cdot \nabla (\gamma^{-1/2} \phi \psi) \\ 
&= \frac{1}{4} \int |\nabla \log \gamma|^2 \phi \psi \, - \frac{1}{2} \int \nabla \log \gamma \cdot \nabla (\phi \psi) 
\end{align*}
for $\phi, \psi \in H^1_\mathrm{loc}(\R^n)$, or in the distributional sense, taking $\phi=1$. Using  again that $\Lambda_{\gamma_1}=\Lambda_{\gamma_2}$, an Alessandrini type identity can be deduced;
\begin{equation}\label{ale}
\big\langle (q_1-q_2)v_1,v_2\big\rangle=0,
\end{equation}
where $q_j$ is the multiplication map associated to $\gamma_j$, and $(-\Delta + q_j)v_j = 0$ in $\Omega$. 

Calder\'on's original idea~\cite{C} was to produce enough oscillatory solutions so that an identity of the type \eqref{ale} would  imply  $q_1=q_2$. 
This was performed, for smooth conductivities, by Sylvester and Uhlmann~\cite{SU0,SU,SU2} using their complex geometrical optics (CGO) 
solutions (see also~\cite{NSU, B,PPU, BT, GLU}). 
That is to say, solutions to $(-\Delta + q_j)v_j = 0$ of the form
\begin{equation}\label{CGOs}v_{\zeta_j} = e^{\zeta_j \cdot x} (1 + w_{\zeta_j})\end{equation}
with $\zeta_j \in \C^n$ and $\zeta_j \cdot \zeta_j = 0$. As $e^{\zeta_j \cdot x}$ is a solution to Laplace's equation, these should be considered to be perturbed solutions with~$w_{\zeta_j}$ small in some sense. Moreover, one can choose different pairs $\zeta_1,\zeta_2 \in \C^n$ so that $$e^{\zeta_1 \cdot x}e^{\zeta_2 \cdot x}=e^{-ik\cdot x}$$ for a fixed frequency $k\in\R^n$. Solutions can be generated so that $w_{\zeta_j}$ decays to zero as $|\zeta_j|\to \infty$, 
in a suitable sense, allowing us to conclude from \eqref{ale} that $q_1=q_2$ by Fourier inversion. A key idea in the work of Haberman and Tataru~\cite{HT} is that it is enough to show that this decay occurs after averaging in $\zeta_j$. 

We have made no attempt to be exhaustive in this very brief history  of the problem. In particular we have  completely neglected reconstruction, stability and numerical results, as well as closely related inverse problems. We recommend~\cite{GKLU} for a more comprehensive bibliography as well as a more gentle introduction to the subject.

It remains to prove the existence of  CGO solutions for the whole class of Lipschitz conductivities.
For this we ask only that they solve the equation in $\Omega$, and not in the whole space as in~\cite{HT}.
That is, substituting \eqref{CGOs} into the Schr\"odinger equation, it is enough to find solutions $w_\zeta$ such that
\begin{equation}
(-\Delta - 2 \zeta \cdot \nabla + q)w_\zeta = - q\qquad \text{in}\ \Omega.
\label{eq:CGOs}
\end{equation} 
In order to achieve this, we prove that the formal adjoint is injective via the a priori estimate\footnote{We  write $a\lesssim b$ whenever $a$ and $b$ are nonnegative quantities that satisfy $a \leq C b$ for a  constant $C > 0$ independent of $M$ and $\tau$. We also  write $a\sim b$ whenever $a\lesssim b$ and $b\lesssim a$. }
\begin{equation}
\|\psi\|_{X_\zeta^{1/2}} \lesssim \| (-\Delta + 2 \zeta \cdot \nabla + q) \psi \|_{X_\zeta^{-1/2}} 
\label{es:apriori}
\end{equation}
for  all $\psi$ in the Schwartz class $\mathcal{S}(\R^n)$ such that $\supp \psi \subset \Omega$. Here, the norms are defined by
\[\| \psi\|_{X^b_\zeta} = \Big( \int (|\zeta| + |p_\zeta (\xi)|)^{2b} |\widehat{\psi}(\xi)|^2 \, d\xi \Big)^{1/2},\]
where $\,\widehat{\psi}\,$ is the Fourier transform of $\psi$ and
$p_\zeta (\xi) = |\xi|^2 +2i \zeta \cdot \xi.$
These spaces, first considered by Haberman and Tataru,  are adapted to the structure of the equation in the spirit of Bourgain spaces; see for example~\cite{T}.
 
A different approach was employed in~\cite{HT}, and the estimate~\eqref{es:apriori} was never stated there,   however it follows from their Lemma 2.2, Corollary 2.1 and Lemma 2.3 for $\|\nabla \log \gamma\|_{L^\infty}$ sufficiently small. Recalling that
$$
(-\Delta + 2 \zeta \cdot \nabla + q)=e^{x\cdot\zeta}(-\Delta+q)e^{-x\cdot\zeta},
$$
in order to remove the smallness condition, we introduce convex Carleman weights in the spirit of~\cite{KSU}, \cite{DKSU}, or \cite{KU}. That is we conjugate with another exponential,
 but this time with quadratic phases multiplied by a large parameter $M>0$. This allows us to improve the control on the~$L^2$ part of  the ${X^{1/2}_\zeta}$-norm. 

With a view to reconstruction, our solutions can be plugged into the usual integral formula which, after averaging, should lead to a reconstruction formula as in~\cite{GZ}. This is in contrast with~\cite{H}, where a sequence of CGO solutions, with good decay properties,  is shown to exist, however it is not so clear which values of~$\zeta$ should be taken in order to obtain this good behaviour. 

In the following section we will prove the existence of the CGO solutions. In the third section we establish Theorem~\ref{th:main_theorem}. In the final section we will prove the local Carleman type estimate that we use to obtain the CGO solutions.% This estimate is sufficiently flexible to treat closely related inverse problems associated to first order perturbations of the conductivity or Schr\"odinger equation~\cite{CPR}.

\begin{acknowledgements} The first author would like to thank Andoni Garc\'ia and Lassi P\"aiv\"arinta for helpful discussions. The authors also thank the anonymous referees for helpful remarks.
\end{acknowledgements}

\section{Existence of the CGO solutions}
We use the method of {\it a priori estimates}, which produces solutions to an equation $Lu = 0$ on a open set $\Omega$ given an a priori estimate for the adjoint of $L$. In the first part of this section, we derive such an a priori estimate. Then we will show how this yields the solutions.

We introduce some notations and an appropriate family of spaces. The Fourier multiplier $ m$ is defined, once and for all, by
\[ m(\xi) = \big(M^{-1} \big||\xi|^2 - \tau^2\big|^2 + M^{-1} \tau^2 |\xi_n|^2 + M\tau^2\big)^{1/2},\]
where $M, \tau >1$ and we write 
\[\widehat{ m(D)^s u}(\xi) =  m(\xi)^s \widehat{u}(\xi),\qquad s\in\R.\]
For  $u$ in the Schwartz class  $\mathcal{S}(\R^n)$,
the Fourier transform  is given by\[\widehat{u}(\xi) = \frac{1}{(2\pi)^{n/2}} \int_{\R^n} e^{-i\xi \cdot x} u(x) \, dx.\]
 For every $u \in \mathcal{S}(\R^n)$, we define the norm
\[\| u \|_{Y^s} = \|  m(D)^s u \|_{L^2}.\]
Later we will take $\sqrt{2} \tau=|\zeta|$ and see how this norm relates with the Haberman--Tataru norm as described in the introduction.

The key ingredient  is the following inequality  which will be proved in the final section.

\begin{theorem} \label{th:carlemanY1/2Y-1/2} \sl Let $R, M, \tau>1$ and set $\varphi (x) = \tau x_n + M x_n^2/2$.
 Then, there is an absolute constant $C$ such that, if $M>CR^2$, then
\[ \| u \|_{Y^{1/2}} \le CR \| e^\varphi (- \Delta ) (e^{-\varphi} u) \|_{Y^{-1/2}} \]
provided  $u \in \mathcal{S}(\R^n)$ with $\supp\, u \subset \{ |x_n| \leq R \}$ and $\tau > 8 M R$.
\end{theorem}

From now on we suppose that $\overline{\Omega}\subset \{x\,:\, |x|<R\}$. By a Whitney extension, we will see in the next section that  we can take  $\gamma \in W^{1, \infty} (\R^n)$, the Sobolev space with first partial derivatives in $L^\infty(\R^n)$, in such a way that $\gamma(x)=1$ whenever $|x|>R$.

The first step in the proof of the a priori estimate is to perturb Theorem~\ref{th:carlemanY1/2Y-1/2}, replacing  $e^\varphi (- \Delta)(e^{-\varphi}u)$ with $e^\varphi (- \Delta + T^\ast q)(e^{-\varphi}u)$, where $T^\ast q$ is defined by
\[\big\langle T^\ast q u, v \big\rangle = \frac{1}{4} \int |\nabla \log T^\ast \gamma|^2 u v  - \frac{1}{2} \int \nabla \log T^\ast \gamma \cdot \nabla (u v) \,.\]
Here, $T^\ast \gamma(x) = \gamma(T x)$ with $T$ a rotation to be chosen later.  By the triangle inequality, if $M > C R^2$, then
\begin{equation}
\| u \|_{Y^{1/2}} \le CR \Big(\| e^\varphi (- \Delta + T^\ast q)(e^{-\varphi}u) \|_{Y^{-1/2}} + \| T^\ast q u \|_{Y^{-1/2}}\Big)
\label{es:+-qu}
\end{equation}
for all $u\in \mathcal{S}(\R^n)$ such that $\supp u \subset \{ |x_n| \leq R \}$ and all $\tau > 8MR$. We now show that the second term on the right-hand side is negligible. 

We compute $\| T^\ast q u \|_{Y^{-1/2}}$ by duality. For all $v \in \mathcal{S} (\R^n)$, we have
\begin{align*}
| \big\langle T^\ast q u, v \big\rangle | & \lesssim  \|\nabla \log \gamma\|^2_{L^\infty} \| u \|_{L^2} \| v \|_{L^2} \\
& \quad + \| \nabla \log \gamma \|_{L^\infty} \big( \| \nabla u \|_{L^2} \| v \|_{L^2} + \| u \|_{L^2} \| \nabla v \|_{L^2} \big).
\end{align*}
Now given that $\| v \|_{L^2} \leq M^{-1/4} \tau^{-1/2} \| v \|_{Y^{1/2}}$
and
\begin{align*}
\| \nabla v \|_{L^2} & \le \Big( \int_{|\xi| < 2\tau} |\xi|^2 |\widehat{v}(\xi)|^2 \, d\xi \Big)^{1/2} + \Big( \int_{|\xi| \geq 2\tau} |\xi|^2 |\widehat{v}(\xi)|^2 \, d\xi \Big)^{1/2} \\
& \lesssim  \tau \| v \|_{L^2} + \Big( \int_{|\xi| \geq 2\tau} \big||\xi|^2 - \tau^2\big| |\widehat{v}(\xi)|^2 \, d\xi \Big)^{1/2} \\
& \le \tau^{1/2} M^{-1/4} \| v \|_{Y^{1/2}} + M^{1/4} \| v \|_{Y^{1/2}},
\end{align*}
this implies
\begin{align}\nonumber
|\big\langle T^\ast qu, v \big\rangle| & \lesssim   A^2\big(M^{-1/2} \tau^{-1} + M^{-1/2} +  \tau^{-1/2}\big)\| u \|_{Y^{1/2}} \| v \|_{Y^{1/2}} \\
& \le  A^2M^{-1/2} \| u \|_{Y^{1/2}} \| v \|_{Y^{1/2}}.\label{try}
\end{align}
for $\tau > 8 MR$. Here, and throughout this section,  $A>1$ is a constant such that 
$$
\|\nabla \log \gamma\|_{L^\infty}=\|\gamma^{-1}\nabla\gamma\|_{L^\infty}<A.
$$
Note that, for Lipschitz conductivites that are uniformly bounded below by zero, this quantity is always bounded, but not necessarily small. 
 By duality, \eqref{try}  yields
\begin{equation}
\| T^\ast q u \|_{Y^{-1/2}} \le   CA^2M^{-1/2}\| u \|_{Y^{1/2}},
\label{es:normY-1/2_qu}
\end{equation}
where $C$ is an absolute constant, and so we can make this small by taking $M$ sufficiently large.  
Plugging \eqref{es:normY-1/2_qu} into \eqref{es:+-qu}, we obtain
\[\| u \|_{Y^{1/2}} \le CR\big( \| e^\varphi (- \Delta + T^\ast q)(e^{-\varphi}u) \|_{Y^{-1/2}} +  A^2M^{-1/2}\| u \|_{Y^{1/2}}\big).\]
Thus, for $M \ge 4 C^2 R^2A^4$, we can absorb the second term on the right-hand side by half of the left-hand side. We summarise what we have obtained in the  following lemma.
\begin{lemma} \label{lem:estimate_q} \sl  Let $R, M, \tau>1$ and set $\varphi (x) = \tau x_n + M x_n^2/2$.
 Then there is an absolute constant $C$ such that, if $M \ge CR^2A^4$, then
\begin{equation*}
\| u \|_{Y^{1/2}} \le CR \| e^\varphi (- \Delta + T^\ast q)(e^{-\varphi}u) \|_{Y^{-1/2}}
\end{equation*}
provided  $u \in \mathcal{S}(\R^n)$ with $\supp\, u \subset \{ |x_n| \leq R \}$ and $\tau > 8 M R$.
\end{lemma}

At this point the quadratic part of $\varphi$ and the parameter $M$ have served their purpose and so we fix $M = C R^2A^4$ with the constant from Lemma \ref{lem:estimate_q}, and note that
\[e^\varphi (- \Delta + T^\ast q)(e^{-\varphi}u ) = e^{Mx_n^2/2} (-\Delta + 2 \tau \partial_{x_n} - \tau^2 + T^\ast q) (e^{-Mx_n^2/2} u). \]
Taking $u = e^{Mx_n^2/2} v $ with $v \in \mathcal{S}(\R^n)$ such that $\supp v \subset \{ |x_n| \leq R \}$, we know that
\begin{equation}\label{kl}\| e^{Mx_n^2/2} v \|_{Y^{1/2}} \le CR\| e^{Mx_n^2/2} (-\Delta + 2 \tau \partial_{x_n} - \tau^2 + T^\ast q) v \|_{Y^{-1/2}}.\end{equation}

We will now remove the remaining exponential factors in a crude fashion. By duality, the estimate
\begin{align*}
\| e^{Mx_n^2/2} (-\Delta + 2 \tau \partial_{x_n} &- \tau^2 + T^\ast q) v \|_{Y^{-1/2}}  \lesssim \| (-\Delta + 2 \tau \partial_{x_n} - \tau^2 + T^\ast q) v \|_{Y^{-1/2}}
\end{align*}
would hold if it were true that
\begin{equation}
\| e^{Mx_n^2/2} \chi w \|_{Y^{1/2}} \lesssim \| w \|_{Y^{1/2}}.
\label{es:exp+M}
\end{equation}
Here $\chi (x_n) = \chi_0(x_n/R)$ where $\chi_0  \in C^\infty_0(\R; [0, 1])$ is such that 
$\chi_0(t) = 1$ for $|t| \leq 2$ and $\chi_0(t) =0 $ for $ |t| > 4$.
On the other hand, the estimate
\[\| v \|_{Y^{1/2}} \lesssim \| e^{Mx_n^2/2} v \|_{Y^{1/2}} \]
would follow from
\begin{equation}
\| e^{-Mx_n^2/2} w \|_{Y^{1/2}} \lesssim \| w \|_{Y^{1/2}}.
\label{es:exp-M}
\end{equation}
Using the following lemma we see that \eqref{es:exp+M} and \eqref{es:exp-M} hold allowing us to remove the exponential factors in \eqref{kl}.
\begin{lemma} \sl Let $f \in \mathcal{S}(\R)$ be a function of the $x_n$ variable. Then
\[\| f u \|_{Y^{1/2}} \lesssim \| p\hat{f} \|_{L^1(\R)} \| u \|_{Y^{1/2}}\]
provided $u \in \mathcal{S}(\R^n)$ and $\tau > M>1$. Here $p(\sigma) = (M^{-1} |\sigma| + 1)^2$.
\end{lemma}
\begin{proof}
Firstly note that it is enough to prove that
\[\|  m(D)^{1/2} (f  m(D)^{-1/2} v)  \|_{L^2} \lesssim \| p\hat{f} \|_{L^1(\R)} \| v \|_{L^2}\]
for all $v \in \mathcal{S}(\R^n)$. Furthermore, note that
\[\mathcal{F}'( m(D)^{-1/2} v) (\xi', x_n) = \mathcal{F}_n^{-1}( m^{-1/2} \widehat{v}) (\xi', x_n)\]
where $\mathcal{F}'$ denotes the Fourier transform in $x'=(x_1, \dots, x_{n - 1})$ with dual variable $\xi'=(\xi_1, \dots, \xi_{n - 1})$ and $\mathcal{F}_n^{- 1}$ denotes the inverse Fourier transform in $\xi_n$ with spacial variable~$x_n$.  
Writing the Fourier transform of a product  as a convolution, 
\begin{align*}
\mathcal{F}(f  m(D)^{-1/2} v)(\xi) &= \frac{1}{(2\pi)^{1/2}} \int_\R e^{-i\xi_n x_n} f(x_n) \mathcal{F}_n^{-1}( m^{-1/2} \widehat{v}) (\xi', x_n) \, dx_n\\
&= \frac{1}{(2\pi)^{1/2}} \int_\R \hat{f} (\xi_n - \eta_n)  m(\xi', \eta_n)^{-1/2} \widehat{v}(\xi', \eta_n) \, d\eta_n.
\end{align*}
Thus, by Plancherel's identity, it is enough to show that
\begin{align*}
\int  m(\xi) \Big| \int_\R \hat{f} (\xi_n - \eta_n)  m(\xi', \eta_n)^{-1/2} \widehat{v}(\xi', \eta_n) & \, d\eta_n \Big|^2 \, d\xi  \lesssim \| p \hat{f} \|^2_{L^1(\R)} \int |\widehat{v} (\xi)|^2 \,d\xi.
\end{align*}
By the Cauchy-Schwarz inequality and Tonelli's theorem,
\begin{align*}
\int  m&(\xi) \Big| \int_\R \hat{f} (\xi_n - \eta_n)  m(\xi', \eta_n)^{-1/2} \widehat{v}(\xi', \eta_n) \, d\eta_n \Big|^2 \, d\xi \\
&\leq \int_\R |\hat{f} (\eta_n)| \, d\eta_n \int  m(\xi) \int_\R \frac{|\hat{f} (\xi_n - \eta_n)|}{  m(\xi', \eta_n)} |\widehat{v}(\xi', \eta_n)|^2 \, d\eta_n \, d\xi \\
& \leq \int_\R |\hat{f} (\eta_n)| \, d\eta_n \sup_{(\xi',\eta_n) \in \R^n} \int_\R \frac{ m(\xi',\xi_n)}{ m(\xi', \eta_n)} |\hat{f} (\xi_n - \eta_n)| \, d\xi_n \, \| \widehat{v} \|^2_{L^2}.
\end{align*}
Thus, we would be done, if we could prove that
\begin{equation}
\frac{ m(\xi',\xi_n)}{ m(\xi', \eta_n)} \lesssim  (M^{-1}|\eta_n - \xi_n| + 1)^2.
\label{es:claim1}
\end{equation}
To see that \eqref{es:claim1} holds, we note that
\[\frac{ m(\xi',\xi_n)}{ m(\xi', \eta_n)} \sim \frac{||\xi'|^2 + |\xi_n|^2 - \tau^2| + \tau |\xi_n| + M \tau}{||\xi'|^2 + |\eta_n|^2 - \tau^2| + \tau |\eta_n| + M \tau}\]
and divide the study of the the quotient in different cases. 
\vspace{1em}

First we suppose that $|\xi_n| \geq 2 \tau$. If $|\xi'| \geq 2 \tau$, then
\[\frac{ m(\xi',\xi_n)}{ m(\xi', \eta_n)} \lesssim \frac{|\xi'|^2 + |\xi_n|^2 + M \tau}{|\xi'|^2 + |\eta_n|^2 + M \tau} \lesssim 1 + \frac{|\xi_n - \eta_n|^2}{M\tau}. \]
 If $|\xi'| < 2 \tau$ and $|\eta_n| \geq 2 \tau$, then the same inequality holds. 
 Finally if $|\xi'| < 2 \tau$ and $|\eta_n| < 2 \tau$, then
\begin{align*}
\frac{ m(\xi',\xi_n)}{ m(\xi', \eta_n)} &\lesssim \frac{|\xi'|^2 + |\xi_n|^2 + M \tau}{||\xi'|^2 + |\eta_n|^2 - \tau^2 | + |\eta_n|^2 + M \tau} \\
& \lesssim \frac{|\xi_n|^2 + M \tau}{|\eta_n|^2 + M \tau} \lesssim 1 + \frac{|\xi_n - \eta_n|^2}{M\tau}.
\end{align*}

Second we suppose that $|\xi_n| < 2 \tau$. If $|\xi'| \geq 2 \tau$, then
\[\frac{ m(\xi',\xi_n)}{ m(\xi', \eta_n)} \lesssim \frac{|\xi'|^2 + |\xi_n|^2 + \tau |\xi_n| + M \tau}{|\xi'|^2 + |\eta_n|^2 + \tau |\eta_n| + M \tau} \lesssim \frac{|\xi'|^2 + M \tau}{|\xi'|^2 + M \tau} \lesssim 1. \]
Finally, if $|\xi_n| < 2 \tau$ and $|\xi'| < 2 \tau$, then
\[\frac{ m(\xi',\xi_n)}{ m(\xi', \eta_n)} \lesssim \frac{\tau |(|\xi'|^2 + |\xi_n|^2)^{1/2} - \tau| + \tau |\xi_n| + M\tau}{\tau |(|\xi'|^2 + |\eta_n|^2)^{1/2} - \tau| + \tau |\eta_n| + M\tau}.\]
Since $|(|\xi'|^2 + |\xi_n|^2)^{1/2} - \tau| \leq |(|\xi'|^2 + |\eta_n|^2)^{1/2} - \tau| + |\xi_n - \eta_n|$, we have that
\[\frac{ m(\xi',\xi_n)}{ m(\xi', \eta_n)} \lesssim 1 + \frac{|\xi_n - \eta_n|}{M}.\]
Therefore, for $\tau > M>1$, we have completed the proof of \eqref{es:claim1}.
\end{proof}
Combining  the inequalities \eqref{kl}, \eqref{es:exp+M} and \eqref{es:exp-M}, we obtain the following lemma.
\begin{lemma} \label{lem:pre-a_priori} \sl There is an absolute constant $C$ such that, for $M = CR^2A^4$, 
\[\| u \|_{Y^{1/2}} \lesssim \| (-\Delta + 2 \tau \partial_{x_n} - \tau^2 + T^\ast q) u \|_{Y^{-1/2}} \]
provided  $u \in \mathcal{S}(\R^n)$ with $\supp\, u \subset \{ |x_n| \leq R \}$ and $\tau > 8 M R$.
 The implicit constant depends on $A$ and $R$.
\end{lemma}
From this estimate, we derive now the a priori estimate that we will use in the second part of this section to construct the CGO solutions.

\begin{proposition} \sl Let $\gamma \in W^{1, \infty} (\R^n)$ satisfy $\gamma (x) = 1$ whenever $|x| > R$  and suppose that $\|\nabla \log \gamma\|_{L^\infty}<A$ for some $A>1$. Define
\[\big\langle q u, v \big\rangle = \frac{1}{4} \int |\nabla \log \gamma|^2 u v  - \frac{1}{2} \int \nabla \log \gamma \cdot \nabla (u v) \,.\]
Let $\zeta = \mathrm{Re}\, \zeta + i \mathrm{Im}\, \zeta \in \C^n$ be such that $|\mathrm{Re}\, \zeta| = |\mathrm{Im}\, \zeta| = \tau$ and $\mathrm{Re}\, \zeta \cdot \mathrm{Im}\, \zeta = 0$. 
Then, there is an absolute constant~$C$ such that, for $\tau > C R^3A^4$,
\[\| u \|_{X_\zeta^{1/2}} \lesssim \| (-\Delta + 2 \zeta \cdot \nabla + q) u \|_{X_\zeta^{-1/2}} \]
provided  $u \in \mathcal{S}(\R^n)$ with $\supp u \subset \{ |x| \leq R \}$. The implicit constant depends on $A$ and $R$.
\end{proposition}
\begin{proof}
Let $T$ be a rotation that satisfies $\mathrm{Re}\, \zeta  = \tau T e_n$ and consider a Schwartz function $v(x) = T^\ast w(x) = w(T x)$ with $\supp w \subset \{ |x| \leq R \}$. Then, by Lemma \ref{lem:pre-a_priori}, 
\[\| v \|_{Y^{1/2}} \lesssim \| (-\Delta + 2 \tau \partial_{x_n} - \tau^2 + T^\ast q) v \|_{Y^{-1/2}} \]
for $M = C R^2A^4$ and $\tau > 8MR$. Obviously,
\[(-\Delta + 2 \tau \partial_{x_n} - \tau^2 + T^\ast q) v = T^\ast [ (-\Delta + 2 \mathrm{Re}\, \zeta \cdot \nabla - |\mathrm{Re}\, \zeta|^2 + q) w ], \]
which implies
\[\| T^\ast w \|_{Y^{1/2}} \lesssim \| T^\ast [ (-\Delta + 2 \mathrm{Re}\, \zeta \cdot \nabla - |\mathrm{Re}\, \zeta|^2 + q) w ] \|_{Y^{-1/2}}. \]
Consider now Schwartz $w(x) =  e^{-i \mathrm{Im}\, \zeta \cdot x} u(x)$ with $\supp u \subset \{ |x| \leq R \}$, we see that
\[T^\ast [ (-\Delta + 2 \mathrm{Re}\, \zeta \cdot \nabla - |\mathrm{Re}\, \zeta|^2 + q) w ] = T^\ast [ e^{-i \mathrm{Im}\, \zeta \cdot x} (-\Delta + 2 \zeta \cdot \nabla + q) u ]\]
and
\[\| T^\ast ( e^{-i \mathrm{Im}\, \zeta \cdot x} u) \|_{Y^{1/2}} \lesssim \| T^\ast [ e^{-i \mathrm{Im}\, \zeta \cdot x} (-\Delta + 2 \zeta \cdot \nabla + q) u ] \|_{Y^{-1/2}}. \]
On the other hand,
\[\mathcal{F} T^\ast ( e^{-i \mathrm{Im}\, \zeta \cdot x} f) (\xi) = \mathcal{F} ( e^{-i \mathrm{Im}\, \zeta \cdot x} f) (T \xi) = \widehat{f} (T \xi + \mathrm{Im}\, \zeta) \]
and
\begin{align*}
\int  m(\xi)^{2b} |\widehat{f} (T \xi &+ \mathrm{Im}\, \zeta)|^2 \,d\xi  \sim \int (||\xi|^2 - 2 \mathrm{Im}\, \zeta \cdot \xi|^2 + |\mathrm{Re}\, \zeta \cdot \xi|^2 + |\zeta|^2)^b |\widehat{f} (\xi)|^2 \, d\xi
\end{align*}
for $f \in X_\zeta^b$, and so we are done.
\end{proof}
Now that we have the a priori estimate, we start the construction of the CGO solutions. Essentially, the estimate tells us that $-\Delta +2 \zeta \cdot \nabla + q$ is injective which provides surjectivity of the formal adjoint and thus a solution $w_\zeta$ to 
$$
(-\Delta - 2 \zeta \cdot \nabla + q) w_\zeta=-q.
$$
However our estimate holds only for compactly supported functions, and so we must take some care with the dual spaces. 
Recalling that $\Omega$ is a bounded domain such that $ \overline{\Omega} \subset \{ |x| < R \}$, we define the space
\[X^{-1/2}_{\zeta,c} (\Omega) = \{ u \in X^{-1/2}_\zeta  : \supp u \subset \overline{\Omega} \}\]
with the norm of $X^{-1/2}_\zeta$. The dual of this  can be characterised by\footnote{This works in the same way as for Sobolev spaces; see for example~\cite{JK}.}
\[X^{1/2}_\zeta (\Omega) = \{ u|_\Omega : u \in X^{1/2}_\zeta \} \]
with the norm
\[\| u \|_{X^{1/2}_\zeta (\Omega)} = \inf \{ \| v \|_{X^b_\zeta} : u = v|_\Omega \}.\]
The existence of $w_\zeta$ such that
\[\big\langle (-\Delta - 2 \zeta \cdot \nabla + q) w_\zeta, u \big\rangle = - \big\langle q , u \big\rangle \]
for all $u \in \mathcal{S}(\R^n)$ such that $\supp u \subset \Omega$,  is equivalent to finding $w_\zeta$ so that
\[\big\langle w_\zeta, (-\Delta + 2 \zeta \cdot \nabla + q) u \big\rangle = - \big\langle q , u \big\rangle. \]
For this we prove that there exists a bounded functional $L$ defined on $X^{-1/2}_{\zeta, c} (\Omega)$ such that
$$L [(-\Delta + 2 \zeta \cdot \nabla + q) u] = - \big\langle q , u \big\rangle $$
for all $u \in \mathcal{S}(\R^n)$ such that $\supp u \subset \Omega$, then existence follows by the previously mentioned Riesz representation theorem. 

Indeed, define the linear space
$$\mathcal{L} = \Big\{ (-\Delta + 2 \zeta \cdot \nabla + q) u : u \in \mathcal{S}(\R^n),\, \supp u \subset \Omega \Big\}$$
and the linear functional
\[Lv = - \big\langle q , u \big\rangle, \] 
for all $v \in \mathcal{L}$ such that $v = (-\Delta + 2 \zeta \cdot \nabla + q) u$. By the a priori estimate, this functional is well-defined and bounded;
\begin{equation}
|L v| \leq \| q \|_{X^{-1/2}_\zeta} \| u \|_{X^{1/2}_\zeta} \lesssim \| q \|_{X^{-1/2}_\zeta} \| v \|_{X^{-1/2}_\zeta}.
\label{es:BoundForL}
\end{equation}
Thus, the functional can be continuously extended to the closure of $\mathcal{L}$ in $X^{-1/2}_{\zeta, c} (\Omega)$. Moreover, it can be extended by zero in the orthogonal complement of $\mathcal{L}$. Still denoting these extensions by $L$, we have a bounded linear functional defined on the whole of $X^{-1/2}_{\zeta, c} (\Omega)$. Since $X^{1/2}_\zeta (\Omega)$ is its dual, 
there exists a unique $w_\zeta \in X^{1/2}_\zeta (\Omega)$ such that
\[\big\langle w_\zeta, v  \big\rangle = L v, \qquad \forall\ v \in X^{-1/2}_{\zeta, c} (\Omega), \]
and
\[\| w_\zeta \|_{X^{1/2}_\zeta (\Omega)} \sim \sup_{v \in X^{-1/2}_{\zeta, c} \setminus \{ 0\}}\frac{| Lv |}{\| v \|_{X^{-1/2}_{\zeta, c} (\Omega)}} \lesssim \| q  \|_{X^{-1/2}_\zeta}.\]
In particular,
\[\big\langle w_\zeta, (-\Delta + 2 \zeta \cdot \nabla + q) u  \big\rangle = - \big\langle q , u \big\rangle, \] 
for all $u \in \mathcal{S}(\R^n)$ such that $\supp u \subset \Omega$. Thus we have proven the existence of the   CGO solutions, which we summarise in the following proposition.

\begin{proposition} \label{prop:CGOs} \sl Let $\gamma \in W^{1, \infty} (\R^n)$ satisfy $\gamma (x) = 1$ whenever $|x| > R$  and suppose that $\|\nabla \log \gamma\|_{L^\infty}<A$ for some $A>1$. Define
\[\big\langle q u, v \big\rangle = \frac{1}{4} \int |\nabla \log \gamma|^2 u v  - \frac{1}{2} \int \nabla \log \gamma \cdot \nabla (u v) \,.\]
Then, there is an absolute constant $C$ such that, for every $$\zeta = \mathrm{Re}\, \zeta + i \mathrm{Im}\, \zeta \in \C^n$$ with $|\mathrm{Re}\, \zeta| = |\mathrm{Im}\, \zeta| = \tau$, $\mathrm{Re}\, \zeta \cdot \mathrm{Im}\, \zeta = 0$ and $\tau > C R^3A^4$, there exists $w_\zeta \in X^{1/2}_\zeta (\Omega)$ so that
\[\| w_\zeta \|_{X^{1/2}_\zeta (\Omega)} \lesssim \| q  \|_{X^{-1/2}_\zeta},\]
and such that
$v_\zeta = e^{\zeta \cdot x} (1 + w_\zeta)$ solves  $(-\Delta + q) v_\zeta = 0$ in $\Omega$.
\end{proposition}

\section{Global uniqueness : The proof of Theorem~\ref{th:main_theorem}}
Since the approach is now reasonably standard, we will be brief.  In particular we are able to simply lift and apply the important new averaged estimate due to Haberman and Tataru~\cite{HT}.

Let $\gamma_1$ and $\gamma_2$ be as in the statement of Theorem \ref{th:main_theorem} with $\gamma_j(x) \leq c_0^{-1}$ and $|\gamma_j(x) - \gamma_j(y)| \leq c_0^{-1} |x - y|$ for all $x, y \in \overline{\Omega}\subset \{ |x| < R \}$ and assume that the associated Dirichlet-to-Neumann maps are equal; $\Lambda_{\gamma_1}=\Lambda_{\gamma_2}$. By the boundary determination theorems due to Alessandrini~\cite{Al90} or Brown~\cite{B01}, we know that $\gamma_1|_{\partial \Omega} = \gamma_2|_{\partial \Omega}$. Thus, we can perform a Whitney type extension of $\gamma_1$ and $\gamma_2$ to $\R^n$ such that the extensions, which we continue to denote by  $\gamma_1$ and $\gamma_2$, are equal outside of $\Omega$, equal to one outside the ball of radius $R$, and satisfy
\begin{align*}
 c_0 \lesssim \gamma_j (x) &\lesssim c_0^{-1},\, \qquad \forall\ x \in \R^n,  \\
 |\gamma_j(x) - \gamma_j(y)| &\lesssim c_0^{-1} |x - y|,\, \qquad \forall\ x, y \in \R^n, 
\end{align*}
with $j = 1,2$ (see for example~\cite[Chapter VI]{St}). The implicit constants  only depend on the dimension $n$. Details can be found in~\cite{CaGR} regarding how to get these specific  properties. Such conductivities can be identified with elements of $W^{1,\infty} (\R^n)$ and so we are able to apply Proposition \ref{prop:CGOs} with $\|\nabla \log\gamma\|_{L^\infty}\le Cc_0^{-2}=A$. Considering the associated potentials, we have the Alessandrini type identity
\begin{equation}
\big\langle (q_1 - q_2) v_1 , v_2 \big\rangle = 0,
\label{id:AlessandriniFOR}
\end{equation}
for all $v_j \in H^1_\mathrm{loc}(\R^n)$ such that $(-\Delta + q_j) v_j = 0$ in $\Omega$. The identity  was proved in \cite[Theorem 7]{B} for global solutions  (see also~\cite[Lemma 2.1]{CaGR}) but the same argument gives the identity for our local solutions.

The next step is to plug the CGO solutions constructed in the previous section into \eqref{id:AlessandriniFOR}.  Let $ P $ be a two-dimensional linear subspace orthogonal to $ k \in \R^n$ and let $ \eta \in S:= P \cap \{ x \in \R^n : |x| = 1 \} $. Let $ \theta $ be the unique vector making $ \{ \eta, \theta \} $ a positively oriented orthonormal basis of $ P $, and define
\begin{align*}
\zeta_1 = \tau \eta + i\Big(-\frac{k}{2} + \Big( \tau^2 - \frac{|k|^2}{4} \Big)^{1/2} \theta \Big) \\
\zeta_2 = -\tau \eta + i\Big(-\frac{k}{2} - \Big( \tau^2 - \frac{|k|^2}{4} \Big)^{1/2} \theta \Big)
\end{align*}
for $\tau > |k| $. Clearly $\zeta_1$ and $\zeta_2$ satisfy the conditions of Proposition~\ref{prop:CGOs}. Thus, there exist solutions to $(-\Delta + q_j) v_{\zeta_j} = 0 $ in $\Omega$ that take the form
$v_{\zeta_j} = e^{\zeta_j \cdot x} (1 + w_{\zeta_j})$
with $w_{\zeta_j} \in X^{1/2}_{\zeta_j} (\Omega) $ and
\begin{equation}
\| w_{\zeta_j} \|_{X^{1/2}_{\zeta_j} (\Omega)} \lesssim \| q_j \|_{X^{-1/2}_{\zeta_j}}
\label{es:boundforw_zetaj}
\end{equation}
for $j=1,2$ and $\tau > \max\{ C R^3 / c_0^8, |k|\}$. 

Let $w_j$ denote any extension of  $w_{\zeta_j} \in X^{1/2}_{\zeta_j} (\Omega)$ to $X^{1/2}_{\zeta_j}$,  and write
\begin{equation*}
v_j = e^{\zeta_j \cdot x} (1 + w_j).
\label{id:vj}
\end{equation*}
Then, to see that $w_j$ belongs to $H^1(\R^n)$ we note 
\begin{align*}
\int_{|\xi| < 4 \tau } (1 + |\xi|^2) |\widehat{w_j}(\xi)|^2 \, d\xi &\lesssim \tau \int_{|\xi| < 4 \tau } (|\zeta| + |p_{\zeta_j} (\xi)|) |\widehat{w_j}(\xi)|^2 \, d\xi \\
\int_{|\xi| \geq 4 \tau } (1 + |\xi|^2) |\widehat{w_j}(\xi)|^2 \, d\xi &\lesssim \int_{|\xi| \geq 4 \tau } (|\zeta| + |p_{\zeta_j} (\xi)|) |\widehat{w_j}(\xi)|^2 \, d\xi,
\end{align*}
so that $v_1$ and $v_2$ belong to $H^1_\mathrm{loc}(\R^n)$.
Thus, $v_1$ and $v_2$ can be plugged into \eqref{id:AlessandriniFOR} and we obtain
\begin{equation}
\begin{aligned}
\big\langle q_1 - q_2, e^{-ik\cdot x} \big\rangle &= -\big\langle q_1 - q_2, e^{-ik\cdot x}w_2 \big\rangle - \big\langle q_1 - q_2, e^{-ik\cdot x}w_1 \big\rangle \\
& \quad - \big\langle (q_1 - q_2)w_1, e^{-ik\cdot x}w_2 \big\rangle.
\end{aligned}
\label{id:FourierTrans}
\end{equation}
First we note that
\begin{align*}
| \big\langle (q_1 - q_2)w_1, e^{-ik\cdot x}w_2 \big\rangle | & \lesssim  (c_0^{-4} + c_0^{-2}|k|) \| w_1 \|_{L^2} \| w_2 \|_{L^2} \\
& \quad + c_0^{-2} ( \| \nabla w_1 \|_{L^2} \| w_2 \|_{L^2} + \| w_1 \|_{L^2} \| \nabla w_2 \|_{L^2}).
\end{align*}
For the terms on the  right-hand side of the previous inequality we note
\[\| w_j \|_{L^2} \lesssim \tau^{-1/2} \| w_j \|_{X^{1/2}_{\zeta_j}}\]
and
\begin{align*}
\| \nabla w_j \|_{L^2} & \leq \Big( \int_{|\xi| < 4\tau} |\xi|^2 |\widehat{w_j}(\xi)|^2 \, d\xi \Big)^{1/2} + \Big( \int_{|\xi| \geq 4\tau} |\xi|^2 |\widehat{w_j}(\xi)|^2 \, d\xi \Big)^{1/2} \\
& \lesssim \tau \| w_j \|_{L^2} + \Big( \int_{|\xi| \geq 4\tau} |p_{\zeta_j}(\xi)| |\widehat{w_j}(\xi)|^2 \, d\xi \Big)^{1/2} \\
& \leq \tau^{1/2} \| w_j \|_{X^{1/2}_{\zeta_j}} + \| w_j \|_{X^{1/2}_{\zeta_j}},
\end{align*}
so that
\begin{equation}
\big| \big\langle (q_1 - q_2)w_1, e^{-ik\cdot x}w_2 \big\rangle \big| \lesssim (c_0^{-2} + |k|)^2 \| w_1 \|_{X^{1/2}_{\zeta_1}} \| w_2 \|_{X^{1/2}_{\zeta_2}}.
\label{es:temrw_1w_2}
\end{equation}
On the other hand, using duality\footnote{See (3.17) in~\cite{CaGR} for more details.}
\begin{equation}
\big| \big\langle q_1 - q_2, e^{-ik\cdot x}w_j \big\rangle \big| \lesssim (1 + |k|) \| q_1 - q_2\|_{X^{-1/2}_{\zeta_j}} \| w_j \|_{X^{1/2}_{\zeta_j}}.
\label{es:temrw_j}
\end{equation}
From the identity \eqref{id:FourierTrans} and the inequalities \eqref{es:temrw_1w_2} and \eqref{es:temrw_j}, we get
\begin{align*}
|\big\langle q_1 - q_2, e^{-ik\cdot x} \big\rangle| &\lesssim (c_0^{-2} + |k|)^2 \Big( \| w_1 \|_{X^{1/2}_{\zeta_1}} \| w_2 \|_{X^{1/2}_{\zeta_2}} \\
& \quad + \| q_1 - q_2 \|_{X^{-1/2}_{\zeta_1}} \| w_1 \|_{X^{1/2}_{\zeta_1}}
+ \| q_1 - q_2 \|_{X^{-1/2}_{\zeta_2}} \| w_2 \|_{X^{1/2}_{\zeta_2}} \Big)
\end{align*}
for any extensions $w_1$ and $w_2$ of $w_{\zeta_1}$ and $w_{\zeta_2}$, which implies
\begin{align*}
\big|\big\langle q_1 - q_2, e^{-ik\cdot x} \big\rangle\big| &\lesssim (c_0^{-2} + |k|)^2 \Big( \| w_{\zeta_1} \|_{X^{1/2}_{\zeta_1} (\Omega)} \| w_{\zeta_2} \|_{X^{1/2}_{\zeta_2}(\Omega)} \\
& \quad + \| q_1 - q_2 \|_{X^{-1/2}_{\zeta_1}} \| w_{\zeta_1} \|_{X^{1/2}_{\zeta_1}(\Omega)}
+ \| q_1 - q_2 \|_{X^{-1/2}_{\zeta_2}} \| w_{\zeta_2} \|_{X^{1/2}_{\zeta_2}(\Omega)} \Big).
\end{align*}
Using \eqref{es:boundforw_zetaj}, we conclude that
\begin{align}\nonumber
|\big\langle q_1 - q_2, e^{-ik\cdot x} \big\rangle| &\lesssim (c_0^{-2} + |k|)^2 \Big( \| q_1 \|_{X^{-1/2}_{\zeta_1} (\Omega)} \| q_2 \|_{X^{-1/2}_{\zeta_2}(\Omega)} \\
& + \| q_1 - q_2 \|_{X^{-1/2}_{\zeta_1}} \| q_1 \|_{X^{-1/2}_{\zeta_1}(\Omega)}
+ \| q_1 - q_2 \|_{X^{-1/2}_{\zeta_2}} \| q_2 \|_{X^{-1/2}_{\zeta_2}(\Omega)} \Big).
\end{align}
\label{es:casiFIN}

Haberman and Tataru proved\footnote{See also  (3.16) from~\cite{CaGR}.} in~\cite{HT} that, for  $ \lambda \geq \max(|k|,1) $,
\begin{equation}
\begin{aligned}
\frac{1}{\lambda} \int_S \int_{\lambda}^{2\lambda} & \|q_j\|^2_{X^{-1/2}_{\zeta_j}} \, d \tau \, d l \lesssim \frac{1}{\lambda} c_0^{-8} + \frac{1}{\lambda^{1/2}} R^{n/2} c_0^{-4} \\
& + (1+ |k|^2) \sup_{|y| < 1} \| \nabla \log \gamma_j - \nabla \log \gamma_j (\centerdot - \lambda^{-1/4} y) \|^2_{L^2},
\end{aligned}
\label{es:acabando}
\end{equation}
where $dl$ is the length measure on $S$. Then, integrating ($\frac{1}{\lambda} \int_S \int_{\lambda}^{2\lambda} \, d \tau \, d l$) the inequality \eqref{es:casiFIN}, applying the Cauchy-Schwarz inequality, using \eqref{es:acabando} and letting $\lambda$ go to infinity, we obtain 
\[\big\langle q_1 - q_2, e^{-ik\cdot x} \big\rangle = 0,\qquad \forall\ k \in \R^n,\]
so that $q_1 = q_2$ by Fourier inversion. One can calculate that
\[ - \nabla \cdot \left( \gamma_1^{1/2} \gamma_2^{1/2} \nabla ( \log \gamma_1^{1/2} - \log \gamma_2^{1/2} ) \right) = \gamma_1^{1/2} \gamma_2^{1/2} (q_2 - q_1) \]
(see for example~\cite[pp. 486]{CaGR}), so that $ \log \gamma_1^{1/2} - \log \gamma_2^{1/2} \in H^1(\Omega) $  is a weak solution of
\[ - \nabla \cdot \left( \gamma_1^{1/2} \gamma_2^{1/2} \nabla u \right) = 0\quad \text{in}\ \Omega, \]
with $\log \gamma_1^{1/2}|_{\partial \Omega} = \log \gamma_2^{1/2}|_{\partial \Omega}$. By the uniqueness of this boundary value problem, we can conclude that $\gamma_1 = \gamma_2$ in $\Omega$.

\section{The proof of Theorem~\ref{th:carlemanY1/2Y-1/2}}\label{sec:a_priori}
 The proof consists of two steps. In the first, we use integration by parts to get a similar estimate, but with the spaces  $Y^1$ and $L^2$. In the second step, we use pseudo-locality of the Fourier multiplier operator and bound a commutator in order to lower the \lq regularity' of the estimate so that it involves the spaces  $Y^{1/2}$ and $Y^{-1/2}$.

The  first step is contained in the following proposition.
\begin{proposition} \label{prop:carlemanY1L2} \sl Let $R, M, \tau>1$ and set $\varphi (x) = \tau x_n + M x_n^2/2$. Then there is an absolute constant $C$ such that 
\[ \| u \|_{Y^1} \le CR \| e^\varphi (- \Delta ) (e^{-\varphi} u) \|_{L^2} \]
provided  $u \in \mathcal{S}(\R^n)$ with $\supp\, u \subset \{ |x_n| \leq R \}$ and $\tau > 2 M R$.
\end{proposition}
\begin{proof}
We start by writing 
\[e^\varphi (- \Delta ) (e^{-\varphi} u) = - \Delta u + \nabla \varphi \cdot \nabla u + \nabla \cdot (\nabla \varphi u) - |\nabla \varphi|^2 u.\]
Defining the formally self-adjoint $A$ and skew-adjoint $B$ by
\[A u = - \Delta u - |\nabla \varphi|^2 u, \qquad B u = \nabla \varphi \cdot \nabla u + \nabla \cdot (\nabla \varphi u), \]
 and integrating by parts, we see that
\begin{equation}
\| e^\varphi (- \Delta ) (e^{-\varphi} u) \|^2_{L^2} = \| Au \|^2_{L^2} + \| Bu \|^2_{L^2} + \int [A, B] u \overline{u} \,, \label{id:CARLEMANibp}
\end{equation}
where $[A, B] = AB - BA$ denotes the commutator.

In order to prove the inequality, we first compute the commutator and then estimate the corresponding terms. Using the definition of $\varphi$, we get that
\begin{align*}
Au(x) &= - \Delta u(x) - (\tau + Mx_n)^2 u(x), \\
Bu(x) &= 2(\tau + Mx_n) \partial_{x_n} u(x) + M u(x),
\end{align*}
which yields
\[[A, B]u(x) = - 4 M \partial^2_{x_n} u(x) + 4M (\tau + Mx_n)^2 u(x).\]
From this expression for the commutator together with a simple integration by parts, we see that\footnote{The choice of $\varphi$ was designed to make this commutator positive.}
\[\int [A, B] u \overline{u}  = 4 M \int |\partial_{x_n} u|^2  + 4 M \int |\nabla \varphi|^ 2 |u|^2 \,.\]
Neglecting the first term on the right-hand side of the previous identity and noting 
\begin{equation}
| \nabla \varphi (x) | \geq \tau - MR \label{es:gradPHI}
\end{equation}
whenever $|x_n| \le R$, we obtain
\[\int [A, B] u \overline{u}  \geq M \tau^2 \int |u|^2 \]
for $\tau > 2MR$. This inequality combined with \eqref{id:CARLEMANibp} implies
\begin{equation}
\| e^\varphi (- \Delta ) (e^{-\varphi} u) \|^2_{L^2} \geq \| Au \|^2_{L^2} + \| Bu \|^2_{L^2} + M \tau^2 \| u \|^2_{L^2}. \label{es:CARLEMANibp}
\end{equation}

Next, we estimate the first two terms on the right-hand side. The inequality $a^2 + b^ 2 \geq (a + b)^2 / 2$  allows us to write
\[\| Au \|^2_{L^2} \geq \frac{1}{2} \| - \Delta u - \tau^2 u \|^2_{L^2} - \| (\tau^2 - |\nabla \varphi|^2) u \|^2_{L^2}.\]
Since $\tau^2 - |\nabla \varphi (x)|^2 = -M x_n (2\tau + Mx_n)$, we get
\begin{equation*}
\| Au \|^2_{L^2} \geq \frac{1}{2} \| - \Delta u - \tau^2 u \|^2_{L^2} - \frac{25}{4}M^2R^2 \tau^2   \| u \|^ 2_{L^2}
\end{equation*}
for $\tau> 2MR$, which can be rewritten as
\begin{equation}
\frac{1}{25 M R^2}\| Au \|^2_{L^2} \geq \frac{1}{50 M R^2} \| - \Delta u - \tau^2 u \|^2_{L^2} - \frac{M \tau^2}{4} \| u \|^ 2_{L^2}. \label{es:Aterm}
\end{equation}
Again, using $a^2 + b^ 2 \geq (a + b)^2 / 2$ and the inequality \eqref{es:gradPHI}, we see that
\begin{align*}
\| Bu \|^2_{L^2} &\geq 2 \| |\nabla \varphi| \partial_{x_n} u \|^2_{L^2} - M^2 \| u \|^2_{L^2} \geq \frac{1}{2}\tau^2 \| \partial_{x_n} u \|^2_{L^2} - M^2 \| u \|^2_{L^2}
\end{align*}
for $\tau > 2MR$. We rewrite this last estimate as\footnote{Here we are preparing the negative term $M^2 \| u \|^2_{L^2}$ so that it can be absorbed by $M\tau^2 \|u \|^2_{L^2}$. It is to facilitate this calculation that we defined the spaces $Y^s$ as we did. Other choices are possible, however, it is only with this choice that we were able to bound the commutator.} 
\begin{equation}
\frac{1}{4M}\| Bu \|^2_{L^2} \geq \frac{\tau^2}{8M} \| \partial_{x_n} u \|^2_{L^2} - \frac{M}{4} \| u \|^2_{L^2}.
\label{es:Bterm}
\end{equation}
for $\tau > 2MR$.
Finally, plugging \eqref{es:Aterm} and \eqref{es:Bterm} into the inequality \eqref{es:CARLEMANibp}, we conclude
\begin{align*}
\| e^\varphi (- \Delta ) (e^{-\varphi} u) \|^2_{L^2} &\geq  \frac{1}{50 M R^2} \| - \Delta u - \tau^2 u \|^2_{L^2} - \frac{M \tau^2}{4} \| u \|^ 2_{L^2} \\
& \quad + \frac{\tau^2}{8M} \| \partial_{x_n} u \|^2_{L^2} - \frac{M}{4} \| u \|^2_{L^2} + M \tau^2 \| u \|_{L^2} \\
&\geq  \frac{1}{50 R^2} \bigg( \frac{1}{M} \| - \Delta u - \tau^2 u \|^2_{L^2} + \frac{\tau^2}{M} \| \partial_{x_n} u \|^2_{L^2}  + M \tau^2 \| u \|^2_{L^2} \bigg)\\ &=  \frac{1}{50 R^2} \| u \|^2_{Y^1}.
\end{align*}
for $\tau > 2MR$. This completes the proof.
\end{proof}

We proceed with the second step. That is, to shift the estimate to the spaces $Y^{1/2}$ and $Y^{-1/2}$. Let $\chi_0 \in C^\infty_0(\R; [0, 1])$ be such that $\chi_0(t) = 1$ for $|t| \leq 2$ and $\chi_0(t) =0 $ for $ |t| > 4$, and define $\chi(x_n) = \chi_0(x_n / R)$. Then
\[\| u \|_{Y^{1/2}} \leq \| \chi  m(D)^{-1/2} u \|_{Y^1} + \| (1 - \chi)  m(D)^{-1/2} u \|_{Y^1}, \]
so by Proposition \ref{prop:carlemanY1L2}, we have
\begin{equation}
\| u \|_{Y^{1/2}} \lesssim R \| e^\varphi (- \Delta)(e^{-\varphi} \chi  m(D)^{-1/2} u) \|_{L^2} + \| (1 - \chi)  m(D)^{-1/2} u \|_{Y^1}
\label{es:afterAPPLYINGY1L2}
\end{equation}
for $\tau > 8MR$.  It remains to bound the right-hand side of this inequality by a constant multiple of $R\| e^\varphi (- \Delta)(e^{-\varphi} u) \|_{Y^{-1/2}}$ plus negligible terms. 

We begin by noting that, by the Leibniz rule,
\begin{align*}
e^\varphi (- \Delta)(e^{-\varphi} \chi  m(D)^{-1/2} u) &= - \partial_{x_n}^2 \chi  m(D)^{-1/2} u - 2 \partial_{x_n} \chi e^\varphi \partial_{x_n} (e^{-\varphi}  m(D)^{-1/2} u)\\
& \quad + \chi e^\varphi (- \Delta)(e^{-\varphi}  m(D)^{-1/2} u).
\end{align*}
Differentiating again, we obtain
\begin{align*}
\| e^\varphi (-& \Delta)(e^{-\varphi} \chi  m(D)^{-1/2} u) \|_{L^2} \lesssim R^{-2} \| u \|_{Y^{-1/2}} \\
& \quad + \| \partial_{x_n} \chi \partial_{x_n} \varphi  m(D)^{-1/2} u \|_{L^2} + R^{-1} \| \partial_{x_n}  m(D)^{-1/2} u \|_{L^2} \\
& \quad + \| e^\varphi (- \Delta)(e^{-\varphi} u) \|_{Y^{-1/2}} + \big\| [ e^\varphi (- \Delta)(e^{-\varphi} \centerdot),  m(D)^{-1/2} ]u \big\|_{L^2},
\end{align*}
where the commutator is defined as usual.
The first three terms on the right-hand side can be estimated easily as follows.
Firstly \begin{align*}
\| u \|_{Y^{-1/2}} \leq M^{-1/4} \tau^{-1/2} \| u \|_{L^2} \leq M^{-1/2} \tau^{-1} \| u \|_{Y^{1/2}}.
\end{align*}
We will use this simple inequality frequently from now on without further comment. Indeed, it is used to see that the second term satisfies
\begin{align*}
\| \partial_{x_n} \chi \partial_{x_n} \varphi  m(D)^{-1/2} u \|_{L^2} &\leq  R^{-1} (\tau +4MR) \| u \|_{Y^{-1/2}} 
 \\
 &\lesssim  R^{-1} M^{-1/2} \| u \|_{Y^{1/2}},
\end{align*}
for $\tau > 8MR$. Finally,
\begin{align} 
\| \partial_{x_n}  m(D)^{-1/2} u \|_{L^2} &\leq \frac{M^{1/2}}{\tau} \bigg( \int \frac{M^{-1}\tau^2 |\xi_n|^2}{M^{-1/2}\tau |\xi_n|} |\widehat{u}(\xi)|^2 \, d\xi \bigg)^{1/2}\!\!\! \leq \frac{M^{1/2}}{\tau} \| u \|_{Y^{1/2}}.\label{fourr}
\end{align}
Altogether, for $\tau > 8MR$, we have
\begin{align}\nonumber
\| e^\varphi (- \Delta)(&e^{-\varphi} \chi  m(D)^{-1/2} u) \|_{L^2} \lesssim \| e^\varphi (- \Delta)(e^{-\varphi} u) \|_{Y^{-1/2}} \\
& \quad + \big\| [ e^\varphi (- \Delta)(e^{-\varphi} \centerdot),  m(D)^{-1/2} ]u \big\|_{L^2} +  R^{-1} M^{-1/2} \| u \|_{Y^{1/2}}. \label{es:gettingY-1/2}
\end{align}

We now turn our attention to the estimate for the commutator.
\begin{lemma} \label{lem:commutator} \sl 
Let $R, M, \tau>1$ and set $\varphi (x) = \tau x_n + M x_n^2/2$. Then there is an absolute constant $C$ such that
\[\big\| [ e^\varphi (- \Delta)(e^{-\varphi} \centerdot),  m(D)^{-1/2} ]u \big\|_{L^2} \le CM^{-1/2} \| u \|_{Y^{1/2}}\]
provided  $u \in \mathcal{S}(\R^n)$ with $\supp\, u \subset \{ |x_n| \leq R \}$ and $\tau > 8 M R$.
\end{lemma}
\begin{proof}
By definition and a simple computation
\begin{align*}
[ e^\varphi (- \Delta)&(e^{-\varphi} \centerdot),  m(D)^{-1/2} ]u \\
&= e^\varphi (- \Delta)(e^{-\varphi}  m(D)^{-1/2} u) -  m(D)^{-1/2}\big[ e^\varphi (- \Delta)(e^{-\varphi} u)\big]\\
&= - |\nabla \varphi|^2  m(D)^{-1/2} u + 2 \nabla \varphi \cdot \nabla  m(D)^{-1/2} u\\
& \quad +  m(D)^{-1/2} (|\nabla \varphi|^2 u) - 2  m(D)^{-1/2} (\nabla \varphi \cdot \nabla u).
\end{align*}
Then we can separate the commutator into a pair of commutators;
\begin{align*}
\big\| [ e^\varphi (- \Delta)&(e^{-\varphi} \centerdot),  m(D)^{-1/2} ]u \big\|_{L^2} \\
& \lesssim \ \|  m(D)^{-1/2} (|\nabla \varphi|^2 u) - |\nabla \varphi|^2  m(D)^{-1/2} u \|_{L^2}\\
& \quad + \| \nabla \varphi \cdot \nabla  m(D)^{-1/2} u -  m(D)^{-1/2} (\nabla \varphi \cdot \nabla u) \|_{L^2}.
\end{align*}
To estimate them, we note that on the Fourier side
\begin{align*}
\mathcal{F} [  m(D)^{-1/2} , |\nabla \varphi|^2 ] u &=  m^{-1/2} (\tau + i M \partial_{\xi_n})^2 \widehat{u} - (\tau + i M \partial_{\xi_n})^2 (  m^{-1/2} \widehat{u})\\
&= -i2 \tau M \partial_{\xi_n}  m^{-1/2} \widehat{u} + M^2 \partial_{\xi_n}^2  m^{-1/2} \widehat{u}\\
& \quad + 2 M^2 \partial_{\xi_n}  m^{-1/2} \partial_{\xi_n} \widehat{u}
\end{align*}
and
\begin{align*}
\mathcal{F} [\nabla \varphi \cdot \nabla ,  m(D)^{-1/2}] u &= (\tau + i M \partial_{\xi_n}) (i \xi_n  m^{-1/2} \widehat{u}) -  m^{-1/2} (\tau + i M \partial_{\xi_n}) (i \xi_n \widehat{u})\\
 &= - M \partial_{\xi_n}  m^{-1/2} \xi_n \widehat{u}.
\end{align*}
From this we see that
\begin{align}\nonumber
\big\| [ e^\varphi (- \Delta)(e^{-\varphi} \centerdot),  m(D)^{-1/2} ]u \big\|_{L^2} 
& \lesssim  \tau M \| \partial_{\xi_n}  m^{-1/2} \widehat{u} \|_{L^2} + M^2 \| \partial_{\xi_n}^2  m^{-1/2} \widehat{u} \|_{L^2} \\
  +& M^2 \| \partial_{\xi_n}  m^{-1/2} \partial_{\xi_n} \widehat{u} \|_{L^2} + M \|\partial_{\xi_n}  m^{-1/2} \xi_n \widehat{u}\|_{L^2}.
\label{es:commutatorL2}
\end{align}
On the other hand, simple computations show that
\begin{align}
\partial_{\xi_n}  m^{-1/2} &= - \frac{1}{4}  m^{-1/2} \frac{\partial_{\xi_n}  m^2}{ m^ 2}, \label{id:firstDERIV} \\
\partial_{\xi_n}^2  m^{-1/2} &= \frac{5}{16}  m^{-1/2} \frac{(\partial_{\xi_n}  m^2)^2}{ m^4} - \frac{1}{4}  m^{-1/2} \frac{\partial_{\xi_n}^2  m^2}{ m^2} \label{id:secondDERIV}
\end{align}
with
\begin{align*}
\partial_{\xi_n}  m(\xi)^2 &= 4 M^{-1}(|\xi|^2 - \tau^2) \xi_n + 2 M^{-1} \tau^2 \xi_n,\\
\partial_{\xi_n}^2  m(\xi)^2 &= 8 M^{-1} \xi_n^2 + 4 M^{-1}(|\xi|^2 - \tau^2) + 2 M^{-1} \tau^2.
\end{align*}

Now, denoting the characteristic functions of $\{\xi : |\xi| \geq 2 \tau \}$ and $\{\xi : |\xi| < 2 \tau \}$ by $\mathbf{1}_{|\xi| \geq 2 \tau}$ and $\mathbf{1}_{|\xi| < 2 \tau}$ respectively, we get
\begin{align}\nonumber
\frac{|\partial_{\xi_n}  m(\xi)^2|}{ m(\xi)^2}& \lesssim \mathbf{1}_{|\xi| \geq 2 \tau} \frac{|\xi|^3}{|\xi|^4 + M^2 \tau^2} + \mathbf{1}_{|\xi| < 2 \tau} \frac{\tau^2 \big||\xi| - \tau\big| + \tau^2 |\xi_n|}{\tau^2 \big||\xi| - \tau\big|^2 + \tau^2 |\xi_n|^2 + M^2 \tau^2}\\\nonumber
& \leq  \frac{|\xi|^3}{|\xi|^4 + M^2 \tau^2} + \frac{|\xi_n|}{|\xi_n|^2 + M^2} + \frac{\big||\xi| - \tau\big|}{\big||\xi| - \tau\big|^2 + M^2}\\& \lesssim  \frac{1}{M^{1/2}\tau^{1/2}} + \frac{1}{M} \le \frac{1}{M},
\label{es:firstDERIV}
\end{align}
for $\tau > 8MR$. Following the same kind of computations, we have
\begin{align}\nonumber
\frac{|\partial_{\xi_n}^2  m(\xi)^2|}{ m(\xi)^2} & \lesssim \mathbf{1}_{|\xi| \geq 2 \tau} \frac{|\xi|^2}{|\xi|^4 + M^2 \tau^2} + \mathbf{1}_{|\xi| < 2 \tau} \frac{\tau \big||\xi| - \tau\big| + \tau^2}{\tau^2 \big||\xi| - \tau\big|^2 + \tau^2 |\xi_n|^2 + M^2 \tau^2}\\ \nonumber
& \leq  \frac{|\xi|^2}{|\xi|^4 + M^2 \tau^2} + \frac{1}{M^2} + \frac{1}{\tau} \frac{\big||\xi| - \tau\big|}{\big||\xi| - \tau\big|^2 + M^2}\\
& \lesssim  \frac{1}{M \tau} + \frac{1}{M^2}\lesssim \frac{1}{M^2}
\label{es:secondDERIV}
\end{align}
for $\tau > 8MR$. 

We can now estimate the right-hand side of \eqref{es:commutatorL2}. Plugging in \eqref{id:firstDERIV}, \eqref{id:secondDERIV}, \eqref{es:firstDERIV} and \eqref{es:secondDERIV} into \eqref{es:commutatorL2}, we see that
\begin{align*}
&\ \quad \big\| [ e^\varphi (- \Delta)(e^{-\varphi} \centerdot),  m(D)^{-1/2} ]u \big\|_{L^2} \\
& \lesssim  \tau \|  m^{-1/2} \widehat{u} \|_{L^2} + \|  m^{-1/2} \widehat{u} \|_{L^2} + M \|  m^{-1/2} \partial_{\xi_n} \widehat{u} \|_{L^2} + \| m^{-1/2} \xi_n \widehat{u}\|_{L^2} \\
&\lesssim  \tau \| u \|_{Y^{-1/2}} + M^{3/4} R \tau^{-1/2} \| u \|_{L^2} +  \frac{M^{1/2}}{\tau} \| u \|_{Y^{1/2}}\\
&\leq  M^{-1/2} \| u \|_{Y^{1/2}} + M^{1/2} R \tau^{-1} \| u \|_{Y^{1/2}} + \frac{M^{1/2}}{\tau} \| u \|_{Y^{1/2}},
\end{align*}
where in the second inequality we used \eqref{fourr} and $\| \partial_{\xi_n} \widehat{u} \|_{L^2} \le R \| \widehat{u} \|_{L^2}$.
Then, for $\tau > 8MR$, this yields the desired estimate.
\end{proof}

Combining the estimates \eqref{es:afterAPPLYINGY1L2} and \eqref{es:gettingY-1/2}  with Lemma \ref{lem:commutator}, we obtain
\begin{equation}
\begin{aligned}
\| u \|_{Y^{1/2}} &\lesssim  R \| e^\varphi (- \Delta)(e^{-\varphi} u) \|_{Y^{-1/2}} + R M^{-1/2} \| u \|_{Y^{1/2}} \\
& \quad + \| (1 - \chi)  m(D)^{-1/2} u \|_{Y^1}
\end{aligned} 	\label{es:afterCOMMUTATOR}
\end{equation}
for $\tau > 8MR$. The second term on the right-hand side of \eqref{es:afterCOMMUTATOR} can be absorbed into the left-hand side for sufficiently large $M$. In what remains of this section we will show that this is also true of the third term. If the cut-off $\chi$ were on the frequency side this part would be very simple, however we were obliged to work on the spatial side in order to use integration by parts arguments in the previous lemmas.

We begin by noting that by definition, 
\begin{align}\nonumber
M^{1/2}\| (1 - \chi)  m(D)^{-1/2} u \|_{Y^1} &\leq   \| (\Delta + \tau^2) [ (1 - \chi)  m(D)^{-1/2} u ] \|_{L^2} \\\nonumber
& \quad +  \tau \| \partial_{x_n} [ (1 - \chi)  m(D)^{-1/2} u] \|_{L^2} \\
& \quad + M\tau \| (1 - \chi)  m(D)^{-1/2} u \|_{L^2}.\label{es:plY1}
\end{align} 
The first term on the right-hand side of \eqref{es:plY1} can be bounded as
\begin{align*}
 \| (\Delta + \tau^2) [ (1 - &\chi)  m(D)^{-1/2} u ] \|_{L^2} \lesssim   \| \partial_{x_n}^2 \chi  m(D)^{-1/2} u \|_{L^2}\\& \quad +  \| \partial_{x_n} \chi \partial_{x_n}  m(D)^{-1/2} u \|_{L^2} + \| (1 - \chi)(\Delta + \tau^2)   m(D)^{-1/2} u \|_{L^2}\\
&\lesssim  M^{-1/2} \tau^{-1} R^{-2} \| u \|_{Y^{1/2}} + M^{1/2}\tau^{-1}R^{-1}  \| u \|_{Y^{1/2}} \\
& \quad +  \| (1 - \chi)(\Delta + \tau^2)   m(D)^{-1/2} u \|_{L^2},
\end{align*}
where in the second inequality we used \eqref{fourr}.
The second term on the right-hand side of \eqref{es:plY1} 
can be bounded as
\begin{align*}
 \tau  \| \partial_{x_n} [& (1 - \chi)  m(D)^{-1/2} u] \|_{L^2} \\
&\leq   \tau \| \partial_{x_n} \chi  m(D)^{-1/2} u \|_{L^2} +  \tau \| (1 - \chi) \partial_{x_n}  m(D)^{-1/2} u \|_{L^2} \\
&\lesssim  M^{-1/2} R^{-1} \| u \|_{Y^{1/2}} +  \tau \| (1 - \chi) \partial_{x_n}  m(D)^{-1/2} u \|_{L^2}
\end{align*}
Thus, \eqref{es:plY1} becomes
\begin{align}\nonumber\ \quad  M^{1/2}\| &(1 - \chi)  m(D)^{-1/2} u \|_{Y^1} \\\nonumber
&\lesssim  \| (1 - \chi) (\Delta + \tau^2)  m(D)^{-1/2} u \|_{L^2}+  \| (1 - \chi) \tau\partial_{x_n}  m(D)^{-1/2} u \|_{L^2} \\
 \label{es:plY1_2} &\quad +  M  \| (1 - \chi) \tau m(D)^{-1/2} u \|_{L^2} + M^{-1/2}R^{-1} \| u \|_{Y^{1/2}}
\end{align}
for $\tau > 8MR$. Again, the last term in \eqref{es:plY1_2} is  negligible, so it remains to prove that the other three are negligible as well. 

Due to the pseudo-locality of the operators and the fact that the supports of $1 - \chi$ and $u$ are separated, it is tempting to suppose that this should be straightforward. However we need to be careful to avoid growth in~$\tau$ coming from the volume of the manifold $ \{ |\xi'| = \tau, \xi_n = 0 \}$.

\begin{lemma} \label{lem:pseudo-locality} \sl Let $N \in \N$. Then there is a constant $C_N$, depending only on $N$, such that
\begin{gather*}
 \| (1 - \chi) (\Delta + \tau^2)  m(D)^{-1/2} u \|_{L^2} \le C_NR^{-N}M^{-N} \| u \|_{Y^{1/2}}, \\
 \| (1 - \chi) \tau \partial_{x_n}  m(D)^{-1/2} u \|_{L^2} \le C_NR^{-N}M^{-N} \| u \|_{Y^{1/2}}, \\
 \| (1 - \chi) \tau m(D)^{-1/2} u \|_{L^2} \le C_NR^{-N}M^{-N} \| u \|_{Y^{1/2}},
\end{gather*}
provided  $u \in \mathcal{S}(\R^n)$ with $\supp\, u \subset \{ |x_n| \leq R \}$ and $\tau > 8 M R$. 
\end{lemma}

Before embarking on the proof of this, we prove a pair of elementary preliminary lemmas.

\begin{lemma}\label{prep1}\sl Let $k\ge 4$. Then $|\partial_{\xi_n}^k  m(\xi)^{-1/2}|$ is bounded above by a constant multiple,  depending only on $k$, of
\begin{align*}
 m(\xi)^{-1/2} \sum_{\ell\ge k/4}^k\Big( \mathbf{1}_{|\xi| \geq 2 \tau} \frac{|\xi|^3}{|\xi|^4+M^2 \tau^2} 
 + \mathbf{1}_{|\xi| < 2\tau} \frac{\big||\xi| - \tau\big| + |\xi_n| + 1}{\big||\xi| - \tau\big|^2 + |\xi_n|^2 + M^2} \Big)^{\ell}.
\end{align*}
\end{lemma}

\begin{proof}
By calculating, it is easy to show that
\begin{equation}
|\partial_{\xi_n}^k  m(\xi)^{-1/2}| \lesssim  m(\xi)^{-1/2}\sum_{\ell=1}^k \sum_{\beta_1 + \dots + \beta_\ell = k} \frac{|\partial_{\xi_n}^{\beta_1}  m(\xi)^2|}{ m(\xi)^2} \cdots \frac{|\partial_{\xi_n}^{\beta_\ell}  m(\xi)^2|}{ m(\xi)^2},
\label{es:derivatives}
\end{equation}
where $\beta_j=1,\ldots,k$.  Again calculations yield that
\begin{align*}
|\partial_{\xi_n}  m(\xi)^2| &\lesssim M^{-1} \big||\xi|^2 - \tau^2\big| |\xi_n| + M^{-1} \tau^2 |\xi_n|, \\
|\partial_{\xi_n}^{2}  m(\xi)^2| &\lesssim M^{-1} |\xi_n|^2 + M^{-1} \big||\xi|^2 - \tau^2\big| + M^{-1} \tau^2, \\
|\partial_{\xi_n}^{3}  m(\xi)^2| &\lesssim M^{-1} |\xi_n|, \\
|\partial_{\xi_n}^{4}  m(\xi)^2| &\lesssim M^{-1},
\end{align*}
and $\partial_{\xi_n}^{5}  m^2=0$. Thus, the nonzero terms in our sum satisfy $$k=\beta_1 + \dots + \beta_\ell\le 4\ell$$ and so in fact we need only sum over the range
 $k/4\le\ell\le k$.

On the other hand, by adding together our four bounds we trivially have that
\begin{equation*}
\frac{|\partial_{\xi_n}^{j}  m(\xi)^2|}{ m(\xi)^2} \lesssim \frac{\big||\xi|^2 - \tau^2\big| |\xi_n| + \tau^2 |\xi_n| + |\xi_n|^2 + \big||\xi|^2 - \tau^2\big| + \tau^2 + |\xi_n|}{\big||\xi|^2 - \tau^2\big|^2 + \tau^2 |\xi_n|^2 + M^2 \tau^2},
\end{equation*}
for $j = 1, \dots, 4$, and so 
\begin{equation*}
\frac{|\partial_{\xi_n}^{j}  m(\xi)^2|}{ m(\xi)^2} \lesssim \mathbf{1}_{|\xi| \geq 2 \tau} \frac{|\xi|^3}{|\xi|^4 + M^2 \tau^2} + \mathbf{1}_{|\xi| < 2\tau} \frac{\big||\xi| - \tau\big| + |\xi_n| + 1}{\big||\xi| - \tau\big|^2 + |\xi_n|^2 + M^2}.
\end{equation*}
Plugging this  into \eqref{es:derivatives}, we get the result.
\end{proof}

We will also need to integrate the factors which appear in the previous lemma. 

\begin{lemma}\label{prep2}\sl Let $\ell\ge 2$ and $\tau>M>1$. Then there is an absolute constant $C$ such that
$$
\sup_{\xi'\in\R^{n-1}}\int_\R \bigg( \frac{|\xi|^3}{|\xi|^4 + M^2 \tau^2} + \frac{\big||\xi| - \tau\big| + |\xi_n| + 1}{\big||\xi| - \tau\big|^2 + |\xi_n|^2 + M^2} \bigg)^{\ell}d\xi_n\le C M^{-\ell+1}.
$$
\end{lemma}

\begin{proof}
Using the elementary inequalities,
\begin{align*}
\frac{|\xi|^3}{|\xi|^4 + M^2 \tau^2} & \lesssim \frac{|\xi'|^3}{|\xi'|^4 + |\xi_n|^4 + M^2 \tau^2} + \frac{|\xi_n|^3}{|\xi_n|^4 + M^2 \tau^2} \\
& \lesssim \frac{1}{(|\xi_n|^4 + M^2 \tau^2)^{1/4}} + \frac{|\xi_n|^3}{|\xi_n|^4 + M^2 \tau^2}
\end{align*}
and
\[\frac{\big||\xi| - \tau\big|}{\big||\xi| - \tau\big|^2 + |\xi_n|^2 + M^2} \lesssim \frac{1}{(|\xi_n|^2 + M^2)^{1/2}}, \]
we see that the left-hand side of our desired inequality is bounded by a constant multiple of 
\begin{align*}\nonumber
\int_\R \bigg(\frac{1}{|\xi_n|  + M}& + \frac{|\xi_n|^3}{|\xi_n|^4 + M^4}+\frac{1}{(|\xi_n|^2 + M^2)^{1/2}}+ \frac{ |\xi_n| + 1}{|\xi_n|^2 + M^2}\bigg)^{\ell}d\xi_n.
\end{align*}
Then, by the change of variables $\xi_n\to My$, this evidently satisfies the desired bound.
\end{proof}

Combining the two lemmas, and using the fact that $ m(\xi)\ge M^{1/2} \tau$, we obtain the following corollary.

\begin{corollary}\label{prepcor}\sl Let $k\ge 8$ and $\tau>M>1$. Then there is a constant~$C$, depending only on $k$, such that
$$\sup_{\xi'\in\R^{n-1}} \|\partial_{\xi_n}^k  m(\xi',\centerdot)^{-1/2} \|_{L^1(\R)}\le CM^{-k/4+3/4} \tau^{-1/2}.$$
\end{corollary}

With these calculations in hand, we are now ready to prove the lemma.

\begin{proof}[Proof of Lemma~\ref{lem:pseudo-locality}]
Consider  $P$ one of the three Fourier multipliers in the lemma and note that
\[\mathcal{F}' (P(D) u) (\xi', x_n) = \mathcal{F}_n^{-1}( P \widehat{u})(\xi', x_n),\]
where $\mathcal{F}'$ denotes the Fourier transform in $x'$, with dual variable $\xi'$, and $\mathcal{F}_n^{- 1}$ denotes the inverse Fourier transform in $\xi_n$, with spacial variable~$x_n$. By Plancherel's identity applied to $\R^{n - 1}$ and the previous identity, we have
\begin{align*}
\| (1 - \chi) &P(D) u \|^2_{L^2} = \int_\R |(1 - \chi)(x_n)|^2 \int_{\R^{n - 1}} \big|\mathcal{F}_n^{-1}( P \widehat{u})(\xi', x_n)\big|^2 \, d\xi' \, dx_n.
\end{align*}
As the inverse Fourier transform of a product can be written as a convolution, we note that
\[\mathcal{F}_n^{-1}( P\widehat{u})(\xi', x_n) = \frac{1}{(2\pi)^{1/2}} \big\langle \mathcal{F}_n^{-1} P(\xi', \centerdot),\mathcal{F}'u(\xi', x_n - \centerdot) \big\rangle.\]

Since $1 - \chi (x_n)$ vanishes when $|x_n| \leq 2R$, we can focus on the points where $|x_n| > 2R$ and introduce $1 - \chi'(y_n)$ with $\chi' (y_n) = \chi_0 (4 y_n / R)$;
\begin{align*}(1 - \chi(x_n)) \mathcal{F}_n^{-1}&( P \widehat{u})(\xi', x_n)= \frac{(1 - \chi(x_n))}{(2\pi)^{1/2}} \big\langle (1 - \chi') \mathcal{F}_n^{-1} P(\xi',\centerdot), \mathcal{F}'u(\xi',x_n - \centerdot) \big\rangle.\end{align*}
This is because if $|y_n| \leq R$, then $|x_n - y_n| > R$ and $\mathcal{F}'u(\xi',x_n - y_n) = 0$, and if $|y_n| > R$ then $1-\chi'(y_n)=1$.  On the other hand,
\[\mathcal{F}_n^{-1} P(\xi',\centerdot)(y_n) =   i^{d}y_n^{-d}  \mathcal{F}_n^{-1}  \partial_{\xi_n}^d P(\xi',\centerdot)(y_n).\]
Taking $d>1$, we have that
\[g_d = i^{d}y_n^{-d} (1 - \chi') \in L^1(\R), \qquad \| g_d \|_{L^1} \lesssim R^{-d+1}. \]
Then, by Young's inequality  and Plancherel's identity,
\begin{align*}
\|(1 - \chi) P(D)u \|_{L^2}& \lesssim  \Big(\int_{\R^{n-1}}\| g_d\mathcal{F}_n^{-1} \partial_{\xi_n}^d P(\xi',\centerdot) \|^2_{L^1(\R)} \| \mathcal{F}'u(\xi',\centerdot) \|^2_{L^2(\R)} d\xi'\Big)^{1/2} \\
&\!\!\!\!\!\!\!\!\lesssim  \Big(\int_{\R^{n-1}}\| g_d \|^2_{L^1(\R)}  \| \mathcal{F}_n^{-1} \partial_{\xi_n}^d P(\xi',\centerdot) \|^2_{L^\infty(\R)} \| \mathcal{F}'u(\xi',\centerdot) \|^2_{L^2(\R)} d\xi'\Big)^{1/2} \\
&\!\!\!\!\!\!\!\! \lesssim R^{-d+1} \sup_{\xi'\in\R^{n-1}}\| \partial_{\xi_n}^d P(\xi',\centerdot) \|_{L^1(\R)} \| u \|_{L^2}\\
&\!\!\!\!\!\!\!\! \lesssim R^{-d+1}M^{-1/4}\tau^{-1/2} \sup_{\xi'\in\R^{n-1}}\| \partial_{\xi_n}^d P(\xi',\centerdot) \|_{L^1(\R)} \| u \|_{Y^{1/2}}.
\end{align*}
Here the implicit constant depends only on $d$. Thus, it will suffice to prove that, for $d \geq 4N + 3$,
\begin{equation}\label{this}  \sup_{\xi'\in\R^{n-1}} \|\partial_{\xi_n}^d P(\xi',\cdot) \|_{L^1(\R)}\le CM^{-N} \tau^{1/2},  
\end{equation}
with multipliers $P$ given by $(\tau^2 - |\centerdot|^2)  m^{-1/2}$, $ i\tau\xi_n  m^{-1/2}$ or $ \tau m^{-1/2}$. The third multiplier was already dealt with by Corollary~\ref{prepcor} and we will only have to work slightly harder for the other two. 

By the Leibniz rule, the second multiplier can be bounded as
\begin{align}\nonumber
\int_\R |  \partial_{\xi_n}^d  ( i \tau\xi_n &m(\xi)^{-1/2}) | \,d\xi_n\\
&\lesssim \int_\R |\tau\xi_n \partial_{\xi_n}^d m(\xi)^{-1/2}| \,d\xi_n + \int_\R |\tau \partial_{\xi_n}^{d - 1} m(\xi)^{-1/2}| \,d\xi_n.\label{es:derivativeMULT2}
\end{align}
Again a sufficient bound for the second term is already given by Corollary~\ref{prepcor}. For the first term, by Lemma~\ref{prep1}, we have
\begin{align*}
\int_\R |\tau\xi_n \partial_{\xi_n}^{d}& m(\xi)^{-1/2}| \, d\xi_n\lesssim \sum_{l\ge d/4}^{d} \int_\R \frac{\tau|\xi_n|}{m(\xi)^{1/2}}\\
\times & \bigg( \mathbf{1}_{|\xi| \geq 2 \tau} \frac{|\xi|^3}{|\xi|^4 + M^2 \tau^2} + \mathbf{1}_{|\xi| < 2\tau} \frac{\big||\xi| - \tau\big| + |\xi_n| + 1}{\big||\xi| - \tau\big|^2 + |\xi_n|^2 + M^2} \bigg)^{l} \,d\xi_n.
\end{align*}
Thus, by Lemma~\ref{prep2} (with $\ell=l-1$),  in order to obtain the bound required in \eqref{this} it would suffice to prove that
\begin{align*}
\frac{\tau|\xi_n|}{ m(\xi)^{1/2}} \bigg( \mathbf{1}_{|\xi| \geq 2 \tau}  \frac{|\xi|^3}{|\xi|^4 + M^2 \tau^2} + \mathbf{1}_{|\xi| < 2\tau} \frac{\big||\xi| - \tau\big| + |\xi_n| + 1}{\big||\xi| - \tau\big|^2 + |\xi_n|^2 + M^2} \bigg)
\lesssim  \tau^{1/2}.
\end{align*}
Using the fact that $ m(\xi)\ge \tau$ and $2||\xi| - \tau||\xi_n|\le ||\xi| - \tau|^2+ |\xi_n|^2$ this is clear by inspection, and so we are also done with the first term in \eqref{es:derivativeMULT2} and it remains to prove \eqref{this} for the first multiplier.

By the Leibniz rule, the first multiplier can be bounded as 
\begin{align}
\int_\R |  \partial_{\xi_n}^d &( (\tau^2 - |\xi|^2) m(\xi)^{-1/2} ) | \, d\xi_n 
\lesssim  \int_\R \big||\xi|^2 - \tau^2\big|| \partial_{\xi_n}^dm(\xi)^{-1/2}| \, d\xi_n\nonumber\\
 &\label{es:derivativeMULT1}\qquad\qquad+ \int_\R |\xi_n \partial^{d-1}_{\xi_n} m(\xi)^{-1/2}|  \, d\xi_n 
 + \int_\R |\partial^{d-2}_{\xi_n}m(\xi)^{-1/2}| \, d\xi_n.
\end{align}
Again the bound for the third term is already given by Corollary~\ref{prepcor}. The second term is similar to the previous multiplier and is bounded in exactly the same way. For the first term, by Lemma~\ref{prep1}, we have
\begin{align*}
\int_\R \big||\xi|^2 -& \tau^2\big| | \partial_{\xi_n}^{d}  m(\xi)^{-1/2}| \, d\xi_n\lesssim \sum_{l\ge d/4}^{d} \int_\R \frac{\big||\xi|^2 - \tau^2\big|}{ m(\xi)^{1/2}}\\
\times & \bigg( \mathbf{1}_{|\xi| \geq 2 \tau} \frac{|\xi|^3}{|\xi|^4 + M^2 \tau^2} + \mathbf{1}_{|\xi| < 2\tau} \frac{\big||\xi| - \tau\big| + |\xi_n| + 1}{\big||\xi| - \tau\big|^2 + |\xi_n|^2 + M^2} \bigg)^{l} \,d\xi_n.
\end{align*}
Thus, by Lemma~\ref{prep2} (with $\ell=l-1$), and using the fact that  $$\frac{\big||\xi|^2 - \tau^2\big|^{1/2}}{m(\xi)^{1/2}}\le M^{1/4},$$ in order to obtain the bound required in \eqref{this} it will suffice to prove 
\begin{align*}
\big||\xi|^2 - \tau^2\big|^{1/2}\bigg( \mathbf{1}_{|\xi| \geq 2 \tau}& \frac{|\xi|^3}{|\xi|^4 + M^2 \tau^2} 
+ \mathbf{1}_{|\xi| < 2\tau} \frac{\big||\xi| - \tau\big| + |\xi_n| + 1}{\big||\xi| - \tau\big|^2 + |\xi_n|^2 + M^2} \bigg) \lesssim  \tau^{1/2}.
\end{align*}
Using the fact that  $2||\xi| - \tau|^{1/2}|\xi_n|\le ||\xi| - \tau|+ |\xi_n|^2$ this is clear by inspection as before. Thus we are done with the first term on the right-hand side of \eqref{es:derivativeMULT1} which completes the proof of  \eqref{this} for the first multiplier. With this, we have completed the proof of the lemma. 
\end{proof}

Returning  to \eqref{es:plY1_2} and applying Lemma \ref{lem:pseudo-locality} with $N = 1$, we obtain 
\[\| (1 - \chi)  m(D)^{-1/2} u \|_{Y^1} \lesssim R^{-1} M^{-1/2} \| u \|_{Y^{1/2}} \]
for $\tau > 8MR$. With this bound \eqref{es:afterCOMMUTATOR} becomes,
\begin{equation*}
\| u \|_{Y^{1/2}} \lesssim R \| e^\varphi (- \Delta)(e^{-\varphi} u) \|_{Y^{-1/2}} + R M^{-1/2} \| u \|_{Y^{1/2}}
\end{equation*}
for $\tau > 8MR$. The second term on the right-hand side is negligible. Indeed, letting $C$ denote the implicit constant, we see that for $M > 4 C^2 R^2$ we can absorb the term by half of the left-hand side, and the proof is complete.

\end{document}